\documentclass[oneside,reqno,english]{amsart}
\usepackage[T1]{fontenc}
\usepackage[utf8]{inputenc}
\usepackage{xcolor}
\usepackage{babel}
\usepackage{prettyref}
\usepackage{amstext}
\usepackage{amsthm}
\usepackage{amssymb}
\usepackage[pdfusetitle,
 bookmarks=true,bookmarksnumbered=false,bookmarksopen=false,
 breaklinks=false,pdfborder={0 0 0},pdfborderstyle={},backref=false,colorlinks=false]
 {hyperref}
\hypersetup{
 colorlinks=true,citecolor=blue,linkcolor=blue,linktocpage=true}

\makeatletter
\numberwithin{equation}{section}
\numberwithin{figure}{section}

\usepackage{prettyref}

\newrefformat{cor}{Corollary~\ref{#1}}
\newrefformat{subsec}{Section~\ref{#1}}
\newrefformat{lem}{Lemma~\ref{#1}}
\newrefformat{thm}{Theorem~\ref{#1}}
\newrefformat{sec}{Section~\ref{#1}}
\newrefformat{chap}{Chapter~\ref{#1}}
\newrefformat{prop}{Proposition~\ref{#1}}
\newrefformat{exa}{Example~\ref{#1}}
\newrefformat{tab}{Table~\ref{#1}}
\newrefformat{rem}{Remark~\ref{#1}}
\newrefformat{def}{Definition~\ref{#1}}
\newrefformat{fig}{Figure~\ref{#1}}
\newrefformat{claim}{Claim~\ref{#1}}
\newrefformat{assu}{Assumption~\ref{#1}}

\makeatother

\theoremstyle{plain}
\newtheorem{thm}{\protect\theoremname}[section]
\newtheorem{lem}[thm]{\protect\lemmaname}
\newtheorem{prop}[thm]{\protect\propositionname}
\theoremstyle{definition}
\newtheorem{example}[thm]{\protect\examplename}
\theoremstyle{remark}
\newtheorem{rem}[thm]{\protect\remarkname}
\theoremstyle{plain}
\newtheorem{cor}[thm]{\protect\corollaryname}
\providecommand{\corollaryname}{Corollary}
\providecommand{\examplename}{Example}
\providecommand{\lemmaname}{Lemma}
\providecommand{\propositionname}{Proposition}
\providecommand{\remarkname}{Remark}
\providecommand{\theoremname}{Theorem}

\begin{document}
\subjclass[2020]{Primary: 47B65; secondary: 15B48, 47A64, 47J25}
\title[Classification in Active Dimension 2 for WR Dynamics]{Classification in Active Dimension 2 for Weighted Residual Dynamics}
\begin{abstract}
We study weighted residual dynamics associated with a rank-one projection
in finite dimension. The iteration reduces, after finitely many steps,
to a nonlinear recursion on a stabilized active subspace. We prove
that this recursion can be classified when the active dimension is
two: either a transverse reducing direction persists unchanged, or
the coupled part collapses completely. As a consequence, we obtain
a description of the limit in the active two-dimensional case and
identify the threshold beyond which higher-dimensional behavior becomes
more flexible.
\end{abstract}

\author{James Tian}
\address{Mathematical Reviews, 535 W. William St, Suite 210, Ann Arbor, MI
48103, USA}
\email{james.ftian@gmail.com}
\keywords{weighted residual dynamics; shorted operators; positive operators;
rank-one projection; nonlinear recursion; active block; classification}

\maketitle
\tableofcontents{}

\section{Introduction}\label{sec:1}

Let $H$ be a Hilbert space, let $P=\left|u\left\rangle \right\langle u\right|$
be a rank-one orthogonal projection, and let $R\in B\left(H\right)_{+}$.
We consider the weighted residual (WR) map 
\[
\Phi_{P}\left(R\right):=R^{1/2}\left(I-P\right)R^{1/2},
\]
and the associated iteration 
\[
R_{0}:=R,\qquad R_{n+1}:=\Phi_{P}\left(R_{n}\right)\qquad\left(n\ge0\right).
\]
This map has appeared in our earlier work as a basic nonlinear operation
on the cone of positive operators. In \cite{tian2025wr} it was used
as the underlying rank-one residual update in a decomposition scheme
for positive operators, while in \cite{tian2025alt} it appeared as
the elementary step in an alternating weighted residual flow. The
purpose of the present paper is to examine the limit problem for the
single projection dynamics itself.

If $R$ commutes with $P$, then $\Phi_{P}\left(R\right)=\left(I-P\right)R\left(I-P\right)$,
so the update reduces to an ordinary compression. Outside the commuting
regime, however, the square-root conjugation changes the geometry
of the iteration in an essential way. The portion removed at step
$n$ is not determined by a fixed subspace alone, but by the interaction
between the current residual $R_{n}$ and the distinguished direction
$u$ through $R^{1/2}_{n}$. Thus the dynamics does not preserve affine
structure on the positive cone, does not come from a variational principle,
and is not simply a linear compression scheme. This nonlinearity is
what made the weighted residual map useful in earlier work, but it
is also what makes the exact limit harder to track. 

The first step in the analysis is to isolate the part of the operator
on which nontrivial dynamics can survive. In finite dimension, the
supports of the iterates form a decreasing chain and therefore stabilize
after finitely many steps. Once this happens, the iteration restricts
to a fixed active subspace $E$, while the maximal reducing part contained
in $u^{\perp}$ remains unchanged throughout. The problem is therefore
reduced to a nonlinear recursion on the stabilized active block.

A basic point is that the induced recursion on $E$ is not, in general,
another copy of the original one. If $u_{E}$ denotes the orthogonal
projection of $u$ onto $E$, then the compressed dynamics takes the
form 
\[
T_{n+1}=T^{1/2}_{n}\left(I_{E}-\left|u_{E}\left\rangle \right\langle u_{E}\right|\right)T^{1/2}_{n}.
\]
Even though $P$ is a rank-one projection on $H$, the operator 
\[
\left|u_{E}\left\rangle \right\langle u_{E}\right|
\]
need not be a projection on $E$. Thus the stabilized problem is another
weighted recursion, rather than a lower-dimensional repetition of
the original projection dynamics.

The main result of the paper is that this induced weighted recursion
can be classified completely when the active dimension is $2$. In
that case there is a sharp dichotomy. Either the transverse line already
reduces the initial active operator, in which case it persists unchanged,
or the active block is coupled, in which case the weighted residual
dynamics collapses completely on that block. As a consequence, one
obtains a description of the limit when $\dim E=2$.

This is the last rigid case. When the active dimension is larger,
the family of stationary points supported on the transverse subspace
becomes larger, and the limiting geometry is no longer forced into
a single pattern. Accordingly, the aim of this paper is not a higher-dimensional
classification, but a precise reduction to the active block together
with a complete analysis of the two-dimensional regime.

The closest classical background is the theory of shorted operators,
beginning with Kreĭn \cite{MR24574} and developed in the work of
Anderson et al. \cite{MR242573,MR287970,MR356949}; also see \cite{MR1465881,MR2234254,MR2573240,MR938493,MR852902}.
There the central object is the maximal positive operator dominated
by a given operator and supported on a prescribed subspace, and the
analysis is governed by operator order, support constraints, and variational
characterizations. The weighted residual dynamics studied here is
different in kind: it is a nonlinear transformation of the positive
cone, and even for a fixed rank-one projection the induced recursion
on the stabilized support is not, in general, another projection recursion. 

For broader context, the present work may be viewed against three
nearby themes. The finite-dimensional matrix background is the theory
of positive definite matrices and their order structure, congruences,
and matrix inequalities \cite{MR2284176}. A second nearby theme is
the study of products and iterative schemes associated with projections
and related Hilbert space dynamics \cite{MR2580440,MR2903120,MR3145756,MR4310540}.
A third one is the theory of operator means \cite{MR563399,MR401352}.
The weighted residual recursion considered here is related in spirit
to these lines of work, but does not seem to fall directly under any
of them.

The paper is organized as follows. \prettyref{sec:2} proves the reduction
to the stabilized active block. \prettyref{sec:3} studies the induced
weighted recursion on that block and develops the structural identities
needed for the asymptotic analysis. \prettyref{sec:4} gives a classification
in active dimension $2$.

\section{Reduction to the Stabilized Active Block}\label{sec:2}

Fix a rank-one orthogonal projection $P=\left|u\right\rangle \left\langle u\right|$.
For any $R\in B\left(H\right)_{+}$, set 
\[
R_{0}:=R,\qquad R_{n+1}:=R^{1/2}_{n}\left(I-P\right)R^{1/2}_{n}\qquad\left(n\ge0\right).
\]
Then $R_{n}-R_{n+1}=R^{1/2}_{n}PR^{1/2}_{n}\ge0$, so $\left(R_{n}\right)$
is monotone decreasing in the operator order. By Vigier's theorem
(see, for example, \cite{MR493419}), $\left(R_{n}\right)$ converges
in the strong operator topology: 
\[
R_{\infty}:=s\text{-}\lim_{n\to\infty}R_{n}\in B\left(H\right)_{+}.
\]
In finite dimension, this is equivalent to norm convergence. The purpose
of this paper is to determine how $R_{\infty}$ is selected by the
interaction of the initial datum $R$ with the distinguished direction
$u$.

We begin by isolating the part of the dynamics that can remain nontrivial
under iteration. The nonlinear WR dynamics has a basic monotonicity:
kernels grow and supports shrink. In finite dimension this forces
eventual stabilization of the support, and once that stabilization
occurs, the iteration restricts to a fixed subspace on which every
iterate is strictly positive. The original problem is therefore reduced,
after finitely many steps, to the analysis of a positive recursion
on a fixed active block.

This reduction separates the geometry coming from the projection $P$
from the geometry already frozen into the initial datum $R$. The
maximal reducing part contained in $u^{\perp}$ remains unchanged
throughout the iteration, while the complementary active component
carries all of the remaining dynamics. The purpose of this section
is to make that decomposition precise and to identify the induced
recursion on the stabilized support.
\begin{lem}
Assume $\dim H=2$ and $P$ has rank-$1$. Then $R_{n}$ converges
in norm to 
\[
R_{\infty}=\begin{cases}
0, & \text{if }PR\ne RP,\\[4pt]
R^{1/2}\left(I-P\right)R^{1/2}=\left(I-P\right)R\left(I-P\right), & \text{if }PR=RP.
\end{cases}
\]
\end{lem}

\begin{proof}
Choose a unit vector $w$ spanning $\mathrm{ran}\left(I-P\right)$.
Then 
\[
I-P=\left|w\left\rangle \right\langle w\right|,\qquad R_{1}=R^{1/2}\left(I-P\right)R^{1/2}=|R^{1/2}w\left\rangle \right\langle R^{1/2}w|.
\]
If $R^{1/2}w=0$, then $R_{1}=0$, hence $R_{n}=0$ for all $n\ge1$.

Otherwise write $R_{1}=\lambda\left|v\left\rangle \right\langle v\right|$
with $\lambda>0$, $\left\Vert v\right\Vert =1$. Then $R^{1/2}_{1}=\sqrt{\lambda}\left|v\left\rangle \right\langle v\right|$,
so 
\[
R_{2}=R^{1/2}_{1}\left(I-P\right)R^{1/2}_{1}=\lambda\left\langle v,\left(I-P\right)v\right\rangle \left|v\left\rangle \right\langle v\right|.
\]
Inductively, 
\[
R_{n}=\lambda\alpha^{n-1}\left|v\left\rangle \right\langle v\right|\qquad\left(n\ge1\right),
\]
where 
\[
\alpha:=\left\langle v,\left(I-P\right)v\right\rangle \in\left[0,1\right].
\]
Hence $R_{n}\to0$ if $\alpha<1$, while $R_{n}\equiv R_{1}$ if $\alpha=1$.

Now 
\[
\alpha=1\iff Pv=0\iff v\in\mathrm{ran}\left(I-P\right)=\mathbb{C}w.
\]
Since $v$ is a scalar multiple of $R^{1/2}w$, this is equivalent
to 
\[
R^{1/2}w\in\mathbb{C}w,
\]
hence to $w$ being an eigenvector of $R$. As $R$ is selfadjoint,
this is equivalent to invariance of $\mathrm{ran}\left(I-P\right)$,
i.e. to 
\[
PR=RP.
\]
Thus, if $PR\ne RP$, then $\alpha<1$ and $R_{\infty}=0$. If $PR=RP$,
then 
\[
R_{1}=R^{1/2}\left(I-P\right)R^{1/2}=\left(I-P\right)R\left(I-P\right),
\]
and therefore $R_{n}\equiv R_{1}$ for all $n\ge1$. 
\end{proof}

\begin{lem}
Let $H$ be any Hilbert space and $P=\left|u\left\rangle \right\langle u\right|$.
Let $\mathcal{M}$ be the largest $R$-reducing subspace contained
in $u^{\perp}$. Then $\mathcal{M}$ reduces every $R_{n}$, and with
respect to $H=\mathcal{M}\oplus\mathcal{M}^{\perp}$, one has 
\[
R_{n}=R\restriction_{\mathcal{M}}\oplus S_{n}.
\]
In particular, 
\[
R_{\infty}=R\restriction_{\mathcal{M}}\oplus S_{\infty}.
\]
\end{lem}

\begin{proof}
Since $\mathcal{M}$ reduces $R=R_{0}$, it also reduces $R^{1/2}_{0}$.
Because $\mathcal{M}\subseteq u^{\perp}=\ker P$, one has 
\[
P\restriction_{\mathcal{M}}=0,\qquad\left(I-P\right)\restriction_{\mathcal{M}}=I_{\mathcal{M}}.
\]
We claim that $\mathcal{M}$ reduces every $R_{n}$ and that 
\[
R_{n}\restriction_{\mathcal{M}}=R\restriction_{\mathcal{M}}.
\]

This is clear for $n=0$. Assume it holds for some $n$. Since $\mathcal{M}$
reduces $R_{n}$, it also reduces $R^{1/2}_{n}$. Hence 
\[
R_{n+1}=R^{1/2}_{n}\left(I-P\right)R^{1/2}_{n}
\]
leaves both $\mathcal{M}$ and $\mathcal{M}^{\perp}$ invariant, so
$\mathcal{M}$ reduces $R_{n+1}$. Moreover, on $\mathcal{M}$, 
\begin{align*}
R_{n+1} & \restriction_{\mathcal{M}}=R^{1/2}_{n}\restriction_{\mathcal{M}}\left(I-P\right)\restriction_{\mathcal{M}}R^{1/2}_{n}\restriction_{\mathcal{M}}\\
 & =R^{1/2}_{n}\restriction_{\mathcal{M}}I_{\mathcal{M}}R^{1/2}_{n}\restriction_{\mathcal{M}}=R_{n}\restriction_{\mathcal{M}}=R\restriction_{\mathcal{M}}.
\end{align*}
By induction, the claim follows for all $n$.

Therefore, with respect to the orthogonal decomposition $H=\mathcal{M}\oplus\mathcal{M}^{\perp}$,
each $R_{n}$ is block diagonal of the form $R_{n}=R\restriction_{\mathcal{M}}\oplus S_{n}$,
where $S_{n}:=R_{n}\restriction_{\mathcal{M}^{\perp}}$. 

Passing to the limit gives $R_{\infty}=R\restriction_{\mathcal{M}}\oplus S_{\infty}$,
where $S_{\infty}\overset{s}{=}\lim_{n\to\infty}S_{n}$. 
\end{proof}

The next proposition shows that, in finite dimension, the support
subspaces $E_{n}=\mathrm{ran}\left(R_{n}\right)$ eventually stabilize.
It also identifies the induced weighted recursion on the limiting
support.
\begin{prop}
\label{prop:2-3} Assume $\dim H<\infty$. For each $n\ge0$, set
\[
K_{n}:=\ker R_{n},\qquad E_{n}:=K^{\perp}_{n}=\mathrm{ran}\left(R_{n}\right).
\]
Then the following hold. 
\begin{enumerate}
\item The kernels increase and the supports decrease: 
\[
K_{n}\subseteq K_{n+1},\qquad E_{n+1}\subseteq E_{n}\qquad\left(n\ge0\right).
\]
\item There exists $N\ge0$ and a subspace $E\subseteq H$ such that 
\[
E_{n}=E\qquad\left(n\ge N\right).
\]
\item For every $n\ge N$, the subspace $E$ reduces $R_{n}$, one has 
\[
R_{n}=0\oplus T_{n}
\]
with respect to the orthogonal decomposition 
\[
H=E^{\perp}\oplus E,
\]
where 
\[
T_{n}:=R_{n}\restriction_{E}\in B\left(E\right)_{+},
\]
and each $T_{n}$ is strictly positive on $E$. 
\item Let $u_{E}:=P_{E}u\in E$. Then for every $n\ge N$, 
\[
T_{n+1}=T^{1/2}_{n}\left(I_{E}-\left|u_{E}\left\rangle \right\langle u_{E}\right|\right)T^{1/2}_{n}.
\]
\item If $u_{E}=0$, then 
\[
T_{n+1}=T_{n}\qquad\left(n\ge N\right).
\]
If $u_{E}\ne0$, then 
\[
0<\left\Vert u_{E}\right\Vert <1,
\]
and 
\[
\det T_{n+1}=(1-\left\Vert u_{E}\right\Vert ^{2})\det T_{n}\qquad\left(n\ge N\right).
\]
Hence 
\[
\det T_{n}=(1-\left\Vert u_{E}\right\Vert ^{2})^{n-N}\det T_{N}\qquad\left(n\ge N\right),
\]
so in particular 
\[
\det T_{n}\to0.
\]
\end{enumerate}
\end{prop}

\begin{proof}
If $x\in K_{n}$, then $R^{1/2}_{n}x=0$, hence $R_{n+1}x=R^{1/2}_{n}\left(I-P\right)R^{1/2}_{n}x=0$.
Thus $K_{n}\subseteq K_{n+1}$. Taking orthogonal complements gives
$E_{n+1}\subseteq E_{n}$. This proves (1).

Because $\dim H<\infty$, the descending chain $E_{0}\supseteq E_{1}\supseteq E_{2}\supseteq\cdots$
must eventually stabilize. Hence there exist $N\ge0$ and a subspace
$E\subseteq H$ such that $E_{n}=E$, $n\ge N$. This proves (2).

Fix $n\ge N$. Since $E=\mathrm{ran}\left(R_{n}\right)=\ker R^{\perp}_{n}$,
the subspace $E$ reduces $R_{n}$, and with respect to $H=E^{\perp}\oplus E$
one has 
\[
R_{n}=0\oplus T_{n},\qquad T_{n}:=R_{n}\restriction_{E}.
\]
Because $\ker\left(T_{n}\right)=\left\{ 0\right\} $ on the finite-dimensional
space $E$, the operator $T_{n}$ is invertible, hence strictly positive
on $E$. This proves (3).

Now let $n\ge N$, and let $x,y\in E$. Since $R^{1/2}_{n}$ vanishes
on $E^{\perp}$ and maps $E$ into $E$, one has $R^{1/2}_{n}\restriction_{E}=T^{1/2}_{n}$.
Therefore 
\[
\left\langle x,R_{n+1}y\right\rangle =\langle R^{1/2}_{n}x,\left(I-P\right)R^{1/2}_{n}y\rangle=\langle T^{1/2}_{n}x,\left(I-P\right)T^{1/2}_{n}y\rangle.
\]
Since $T^{1/2}_{n}x,T^{1/2}_{n}y\in E$, and $P_{E}$ acts as the
identity on $E$, one gets 
\[
\langle T^{1/2}_{n}x,\left(I-P\right)T^{1/2}_{n}y\rangle=\langle T^{1/2}_{n}x,\left(I_{E}-\left|u_{E}\left\rangle \right\langle u_{E}\right|\right)T^{1/2}_{n}y\rangle.
\]
Hence 
\[
T_{n+1}=T^{1/2}_{n}\left(I_{E}-\left|u_{E}\left\rangle \right\langle u_{E}\right|\right)T^{1/2}_{n}.
\]
This proves (4).

If $u_{E}=0$, then $I_{E}-\left|u_{E}\left\rangle \right\langle u_{E}\right|=I_{E}$,
so $T_{n+1}=T_{n}$, $n\ge N$. 

Assume now that $u_{E}\ne0$. Since $u_{E}=P_{E}u$, one has 
\[
0<\left\Vert u_{E}\right\Vert \le\left\Vert u\right\Vert =1.
\]
Taking determinants in the identity from (4) yields 
\[
\det T_{n+1}=\det T_{n}\cdot\det\left(I_{E}-\left|u_{E}\left\rangle \right\langle u_{E}\right|\right).
\]
Now $\left|u_{E}\left\rangle \right\langle u_{E}\right|$ is rank
one on $E$, with spectrum 
\[
\left\{ \left\Vert u_{E}\right\Vert ^{2},0,\dots,0\right\} ,
\]
hence 
\[
\det\left(I_{E}-\left|u_{E}\left\rangle \right\langle u_{E}\right|\right)=1-\left\Vert u_{E}\right\Vert ^{2}.
\]
Therefore 
\[
\det T_{n+1}=(1-\left\Vert u_{E}\right\Vert ^{2})\det T_{n}\qquad\left(n\ge N\right).
\]
Since each $T_{n}$ is strictly positive on the finite-dimensional
space $E$, one has $\det T_{n}>0$, $n\ge N$. Hence 
\[
1-\left\Vert u_{E}\right\Vert ^{2}=\frac{\det T_{n+1}}{\det T_{n}}>0,
\]
so $\left\Vert u_{E}\right\Vert <1$. Iterating gives 
\[
\det T_{n}=(1-\left\Vert u_{E}\right\Vert ^{2})^{n-N}\det T_{N}\qquad\left(n\ge N\right).
\]
Since $0<1-\left\Vert u_{E}\right\Vert ^{2}<1$, it follows that $\det T_{n}\to0$.
This proves (5). 
\end{proof}

By part (4) of \prettyref{prop:2-3}, the finite-dimensional problem,
after finitely many support drops, reduces to the recursion 
\begin{equation}
T_{n+1}=T^{1/2}_{n}\left(I_{E}-\left|u_{E}\left\rangle \right\langle u_{E}\right|\right)T^{1/2}_{n}\qquad\left(n\ge N\right)\label{eq:2-1}
\end{equation}
on the stabilized support $E$.

If $u_{E}=0$, then the recursion is stationary. If $u_{E}\ne0$,
then $0<\left\Vert u_{E}\right\Vert <1$, so the compressed rank-one
operator $\left|u_{E}\left\rangle \right\langle u_{E}\right|$ is
not a projection on $E$. Thus the stabilized dynamics is not, in
general, the original rank-one projection recursion in lower dimension.

The determinant decay in part (5) shows that the limit is singular,
but it does not determine its rank. The following example shows that
the limit on $E$ may in fact be zero. Thus the stabilized support
need not coincide with the support of the limit.
\begin{example}
\label{exa:2-4} Let $H=\mathbb{C}^{3}$ with orthonormal basis $\left\{ e_{0},e_{1},e_{2}\right\} $,
and set $E:=\mathrm{span}\left\{ e_{1},e_{2}\right\} $. Define 
\[
u:=\frac{1}{\sqrt{2}}e_{0}+\frac{1}{\sqrt{2}}e_{1},\qquad P:=\left|u\left\rangle \right\langle u\right|.
\]
Let 
\[
R_{0}:=0_{\mathbb{C}e_{0}}\oplus T_{0},\qquad T_{0}:=\begin{pmatrix}1 & 1\\
1 & 2
\end{pmatrix}
\]
with respect to the orthogonal decomposition $H=\mathbb{C}e_{0}\oplus E$. 

Then $R_{0}\in B\left(H\right)_{+}$, and the support is already stabilized
at $E$. The induced recursion on $E$ is given by \prettyref{eq:2-1},
where 
\begin{equation}
u_{E}=P_{E}u=\frac{1}{\sqrt{2}}e_{1}\ne0.\label{eq:2-2}
\end{equation}
Moreover, $T_{n}\to0$. Hence $T_{\infty}=0$. 

In particular, although $\mathrm{ran}\left(R_{n}\right)=E$ for every
$n\ge0$, one has $\mathrm{ran}\left(R_{\infty}\right)=\left\{ 0\right\} $.
Thus the eventual support of the iterates need not coincide with the
support of the limit.
\end{example}

\begin{proof}
Since $R_{0}$ vanishes on $\mathbb{C}e_{0}$ and is strictly positive
on $E$, one has 
\[
\ker R_{0}=\mathbb{C}e_{0},\qquad\mathrm{ran}\left(R_{0}\right)=E.
\]
Because $R^{1/2}_{0}$ also vanishes on $\mathbb{C}e_{0}$ and maps
$E$ into itself, the recursion $R_{n+1}=R^{1/2}_{n}\left(I-P\right)R^{1/2}_{n}$
preserves the decomposition $H=\mathbb{C}e_{0}\oplus E$ and reduces
to \prettyref{eq:2-1} with $u_{E}$ given by \prettyref{eq:2-2}.
Thus, with respect to the ordered basis $\left\{ e_{1},e_{2}\right\} $
of $E$, 
\[
I_{E}-\left|u_{E}\left\rangle \right\langle u_{E}\right|=\begin{pmatrix}\frac{1}{2} & 0\\
0 & 1
\end{pmatrix}.
\]

Write 
\[
T_{n}=\begin{pmatrix}a_{n} & b_{n}\\
b_{n} & d_{n}
\end{pmatrix},\qquad a_{0}=1,\quad b_{0}=1,\quad d_{0}=2.
\]
Since each $T_{n}$ is positive definite, define 
\[
y_{n}:=\frac{b_{n}}{d_{n}},\qquad z_{n}:=\frac{\sqrt{a_{n}d_{n}-b^{2}_{n}}}{d_{n}}.
\]
Then 
\[
\frac{a_{n}}{d_{n}}=y^{2}_{n}+z^{2}_{n}.
\]

For a positive $2\times2$ matrix 
\[
M=\begin{pmatrix}a & b\\
b & d
\end{pmatrix},
\]
one has the square-root formula 
\[
M^{1/2}=\frac{M+\sqrt{\det M}\,I}{\sqrt{\mathrm{tr}\left(M\right)+2\sqrt{\det M}}}.
\]
Applying this to 
\[
T_{n+1}=T^{1/2}_{n}\begin{pmatrix}\frac{1}{2} & 0\\
0 & 1
\end{pmatrix}T^{1/2}_{n},
\]
one obtains 
\[
y_{n+1}=y_{n}\frac{y^{2}_{n}+z^{2}_{n}+3z_{n}+2}{y^{2}_{n}+2z^{2}_{n}+4z_{n}+2},
\]
\[
z_{n+1}=\sqrt{2}z_{n}\frac{y^{2}_{n}+z^{2}_{n}+2z_{n}+1}{y^{2}_{n}+2z^{2}_{n}+4z_{n}+2},
\]
and 
\[
d_{n+1}=d_{n}\frac{y^{2}_{n}+2z^{2}_{n}+4z_{n}+2}{2\left(y^{2}_{n}+z^{2}_{n}+2z_{n}+1\right)}.
\]

We first show that $z_{n}\to0$. Set 
\[
f\left(y,z\right):=\sqrt{2}\frac{y^{2}+z^{2}+2z+1}{y^{2}+2z^{2}+4z+2}.
\]
Then 
\[
z_{n+1}=f\left(y_{n},z_{n}\right)z_{n}.
\]
A direct calculation gives 
\[
\frac{\partial f}{\partial z}=-\frac{2\sqrt{2}y^{2}\left(z+1\right)}{\left(y^{2}+2z^{2}+4z+2\right)^{2}}\le0,\qquad\frac{\partial f}{\partial y}=\frac{2\sqrt{2}y\left(z+1\right)^{2}}{\left(y^{2}+2z^{2}+4z+2\right)^{2}}\ge0.
\]
Since 
\[
y_{0}=\frac{1}{2},\qquad z_{0}=\frac{1}{2},
\]
and the recursion keeps $y_{n},z_{n}$ in the interval $\left(0,\frac{1}{2}\right]$,
it follows that 
\[
f\left(y_{n},z_{n}\right)\le f\left(\frac{1}{2},0\right)=\sqrt{2}\frac{1+\frac{1}{4}}{2+\frac{1}{4}}=\frac{5\sqrt{2}}{9}<1.
\]
Hence $z_{n+1}\le\frac{5\sqrt{2}}{9}z_{n}$, so $z_{n}\to0$ geometrically.
In particular, $\sum^{\infty}_{n=0}z_{n}<\infty$. 

Next, $y_{n}$ decreases to a positive limit. Indeed, 
\[
y_{n+1}=y_{n}\left(1-\frac{z^{2}_{n}+z_{n}}{y^{2}_{n}+2z^{2}_{n}+4z_{n}+2}\right),
\]
so $0<y_{n+1}<y_{n}$. Also, 
\[
\frac{z^{2}_{n}+z_{n}}{y^{2}_{n}+2z^{2}_{n}+4z_{n}+2}\le\frac{z_{n}\left(z_{n}+1\right)}{2\left(z_{n}+1\right)^{2}}\le\frac{z_{n}}{2},
\]
hence $y_{n+1}\ge y_{n}\left(1-\frac{z_{n}}{2}\right)$. Because $\sum^{\infty}_{n=0}z_{n}<\infty$,
the infinite product 
\[
\prod^{\infty}_{n=0}\left(1-\frac{z_{n}}{2}\right)
\]
converges to a positive number, so $y_{n}\downarrow y_{\infty}$ for
some $y_{\infty}>0$. 

Finally, we show that $d_{n}\to0$. From the recursion for $d_{n}$,
\[
1-\frac{d_{n+1}}{d_{n}}=\frac{y^{2}_{n}}{2\left(y^{2}_{n}+z^{2}_{n}+2z_{n}+1\right)}.
\]
Since $y_{n}\ge y_{\infty}>0$ and $0<z_{n}\le z_{0}=\frac{1}{2}$,
one gets 
\[
1-\frac{d_{n+1}}{d_{n}}\ge\frac{y^{2}_{\infty}}{2\left(y^{2}_{0}+z^{2}_{0}+2z_{0}+1\right)}=\frac{y^{2}_{\infty}}{5}.
\]
Therefore $d_{n+1}\le\left(1-\frac{y^{2}_{\infty}}{5}\right)d_{n}$,
so $d_{n}\to0$ exponentially.

Since $b_{n}=y_{n}d_{n}$, $a_{n}=\left(y^{2}_{n}+z^{2}_{n}\right)d_{n}$,
it follows that 
\[
a_{n}\to0,\qquad b_{n}\to0,\qquad d_{n}\to0.
\]
Hence $T_{n}\to0$ and therefore $T_{\infty}=0$. 
\end{proof}

\section{The Induced Weighted Recursion}\label{sec:3}

After the support stabilizes, the original projection recursion is
replaced by a new nonlinear iteration on the active block $E$, as
in \prettyref{eq:2-1}. Unlike the original rank-one projection $P=\left|u\left\rangle \right\langle u\right|$,
the compressed operator $\left|u_{E}\left\rangle \right\langle u_{E}\right|$
need not be a projection on $E$. Thus the stabilized dynamics is
not simply a lower-dimensional copy of the original recursion; it
is a weighted version of it. The aim of this section is to extract
structural information from that weighted recursion. We show that
the distinguished direction determined by $u_{E}$ always disappears
from the limit, while the orthogonal block evolves by a monotone rank-one
decrement. These identities explain both the rigidity visible in dimension
$2$ and the additional freedom that appears in higher dimensions.

Assume $u_{E}\ne0$, and write 
\begin{equation}
e:=\frac{u_{E}}{\left\Vert u_{E}\right\Vert },\qquad\tau:=\left\Vert u_{E}\right\Vert ^{2}\in\left(0,1\right).\label{eq:3-1}
\end{equation}
Then 
\begin{equation}
T_{n+1}=T^{1/2}_{n}\left(I_{E}-\tau\left|e\left\rangle \right\langle e\right|\right)T^{1/2}_{n}\qquad\left(n\ge0\right).\label{eq:3-2}
\end{equation}
Set 
\begin{equation}
Q:=I_{E}-\left|e\left\rangle \right\langle e\right|,\label{eq:3-3}
\end{equation}
so that $Q$ is the orthogonal projection of $E$ onto $e^{\perp}$. 
\begin{prop}[Asymptotic block decoupling relative to the defect direction]
\label{prop:3-1} Decompose $T_{n}$ with respect to $E=\mathbb{C}e\oplus e^{\perp}$
as 
\[
T_{n}=\begin{pmatrix}a_{n} & b^{*}_{n}\\
b_{n} & B_{n}
\end{pmatrix},
\]
where $a_{n}:=\left\langle e,T_{n}e\right\rangle \in\mathbb{R}_{+}$,
$b_{n}:=QT_{n}e\in e^{\perp}$, and $B_{n}:=QT_{n}Q\in B(e^{\perp})_{+}$.
Also set $y_{n}:=QT^{1/2}_{n}e\in e^{\perp}$. Then the following
hold. 
\begin{enumerate}
\item For every $n\ge0$, 
\[
a_{n+1}=\left(1-\tau\right)a_{n}+\tau\left\Vert y_{n}\right\Vert ^{2}.
\]
\item One has 
\[
\tau\sum^{\infty}_{n=0}\left\Vert y_{n}\right\Vert ^{2}=\mathrm{tr}\left(B_{0}-B_{\infty}\right)<\infty,
\]
where $B_{\infty}:=QT_{\infty}Q$. In particular, $y_{n}\to0$.
\item One has $a_{n}\to0$.
\item One has $b_{n}\to0$. Consequently, 
\[
T_{\infty}=\begin{pmatrix}0 & 0\\
0 & B_{\infty}
\end{pmatrix}
\]
with respect to $E=\mathbb{C}e\oplus e^{\perp}$. Equivalently, 
\[
T_{\infty}e=0,\qquad\mathrm{ran}\left(T_{\infty}\right)\subseteq e^{\perp}.
\]
\end{enumerate}
\end{prop}

\begin{proof}
Using \prettyref{eq:3-2}, one has 
\[
a_{n+1}=\left\langle e,T_{n+1}e\right\rangle =\langle T^{1/2}_{n}e,\left(I_{E}-\tau\left|e\left\rangle \right\langle e\right|\right)T^{1/2}_{n}e\rangle.
\]
Therefore 
\[
a_{n+1}=\Vert T^{1/2}_{n}e\Vert^{2}-\tau|\langle e,T^{1/2}_{n}e\rangle|^{2}.
\]
Now $a_{n}=\left\langle e,T_{n}e\right\rangle =\Vert T^{1/2}_{n}e\Vert^{2}$,
and since $T^{1/2}_{n}e=\langle e,T^{1/2}_{n}e\rangle e+y_{n}$, $y_{n}=QT^{1/2}_{n}e\in e^{\perp}$,
one gets $|\langle e,T^{1/2}_{n}e\rangle|^{2}=a_{n}-\left\Vert y_{n}\right\Vert ^{2}$,
and so 
\[
a_{n+1}=a_{n}-\tau(a_{n}-\left\Vert y_{n}\right\Vert ^{2})=\left(1-\tau\right)a_{n}+\tau\left\Vert y_{n}\right\Vert ^{2}.
\]
This proves (1).

Next, 
\[
B_{n+1}=QT_{n+1}Q=QT_{n}Q-\tau|QT^{1/2}_{n}e\left\rangle \right\langle QT^{1/2}_{n}e|=B_{n}-\tau\left|y_{n}\left\rangle \right\langle y_{n}\right|.
\]
Taking traces gives 
\[
\mathrm{tr}\left(B_{n}\right)-\mathrm{tr}\left(B_{n+1}\right)=\tau\left\Vert y_{n}\right\Vert ^{2}.
\]
Summing from $n=0$ to $m-1$ yields 
\[
\tau\sum^{m-1}_{n=0}\left\Vert y_{n}\right\Vert ^{2}=\mathrm{tr}\left(B_{0}\right)-\mathrm{tr}\left(B_{m}\right).
\]
Since $B_{m}\to B_{\infty}$ in norm (Vigier's theorem), letting $m\to\infty$
gives 
\[
\tau\sum^{\infty}_{n=0}\left\Vert y_{n}\right\Vert ^{2}=\mathrm{tr}\left(B_{0}-B_{\infty}\right)<\infty.
\]
In particular, $y_{n}\to0$. This proves (2).

Set $\rho:=1-\tau\in\left(0,1\right)$. Iterating the recurrence in
(1) gives 
\[
a_{n}=\rho^{n}a_{0}+\tau\sum^{n-1}_{k=0}\rho^{\,n-1-k}\left\Vert y_{k}\right\Vert ^{2}.
\]
The first term tends to $0$. Since $(\left\Vert y_{k}\right\Vert ^{2})\in\ell^{1}$,
the second term also tends to $0$. Hence $a_{n}\to0$. This proves
(3).

Finally, let $x\in e^{\perp}$. Since $T_{n}\ge0$, the Cauchy-Schwarz
inequality for the positive sesquilinear form $\left(x,z\right)\mapsto\left\langle x,T_{n}z\right\rangle $
gives 
\[
\left|\left\langle x,b_{n}\right\rangle \right|^{2}=\left|\left\langle x,T_{n}e\right\rangle \right|^{2}\le\left\langle x,T_{n}x\right\rangle \left\langle e,T_{n}e\right\rangle =\left\langle x,T_{n}x\right\rangle a_{n}.
\]
Since $0\le T_{n}\le T_{0}$, one has $\left\langle x,T_{n}x\right\rangle \le\left\Vert T_{0}\right\Vert \left\Vert x\right\Vert ^{2}$,
and therefore $\left|\left\langle x,b_{n}\right\rangle \right|^{2}\le\left\Vert T_{0}\right\Vert \left\Vert x\right\Vert ^{2}a_{n}$.
Taking the supremum over all $x\in e^{\perp}$ with $\left\Vert x\right\Vert =1$
yields $\left\Vert b_{n}\right\Vert ^{2}\le\left\Vert T_{0}\right\Vert a_{n}$.
Since $a_{n}\to0$, it follows that $b_{n}\to0$. This proves (4).

Because $T_{n}\to T_{\infty}$ in norm, $a_{n}\to0$, $b_{n}\to0$,
and $B_{n}\to B_{\infty}$, the block form of $T_{\infty}$ follows:
\[
T_{\infty}=\begin{pmatrix}0 & 0\\
0 & B_{\infty}
\end{pmatrix}.
\]
Equivalently, $T_{\infty}e=0$ and $\mathrm{ran}\left(T_{\infty}\right)\subseteq e^{\perp}$. 
\end{proof}

The preceding proposition shows that the distinguished direction always
disappears from the limit. This does not force complete collapse on
the stabilized support. At the opposite extreme, if the orthogonal
complement of the distinguished direction is already reducing for
the initial datum, then the weighted recursion leaves that transverse
block unchanged and only damps the defect direction.
\begin{prop}
\label{prop:3-2} Suppose that $e^{\perp}$ reduces $T_{0}$. Then
$e^{\perp}$ reduces every $T_{n}$, and with respect to $E=\mathbb{C}e\oplus e^{\perp}$
one has 
\[
T_{n}=\lambda_{n}\left|e\left\rangle \right\langle e\right|\oplus B_{0},
\]
where 
\[
\lambda_{0}=\left\langle e,T_{0}e\right\rangle ,\qquad\lambda_{n+1}=\left(1-\tau\right)\lambda_{n}.
\]
Hence 
\[
\lambda_{n}=\left(1-\tau\right)^{n}\lambda_{0},\qquad T_{\infty}=0\oplus B_{0}.
\]
In particular, if $B_{0}\ne0$, then the stabilized weighted recursion
does not collapse to zero. 
\end{prop}

\begin{proof}
Since $e^{\perp}$ reduces $T_{0}$, one has $T_{0}=\lambda_{0}\left|e\left\rangle \right\langle e\right|\oplus B_{0}$
with respect to $E=\mathbb{C}e\oplus e^{\perp}$, where 
\[
\lambda_{0}=\left\langle e,T_{0}e\right\rangle \ge0,\qquad B_{0}:=QT_{0}Q\in B\left(e^{\perp}\right)_{+}.
\]
We claim that for every $n\ge0$, 
\[
T_{n}=\lambda_{n}\left|e\left\rangle \right\langle e\right|\oplus B_{0}
\]
for some scalar $\lambda_{n}\ge0$.

This is clear for $n=0$. Assume it holds for some $n$. Since the
decomposition is reducing for $T_{n}$, it is also reducing for $T^{1/2}_{n}$,
and 
\[
T^{1/2}_{n}=\sqrt{\lambda_{n}}\left|e\left\rangle \right\langle e\right|\oplus B^{1/2}_{0}.
\]
Moreover, 
\[
I_{E}-\left|u_{E}\left\rangle \right\langle u_{E}\right|=I_{E}-\tau\left|e\left\rangle \right\langle e\right|=\left(1-\tau\right)\left|e\left\rangle \right\langle e\right|\oplus I_{e^{\perp}}.
\]
Therefore 
\[
T_{n+1}=T^{1/2}_{n}\left(I_{E}-\tau\left|e\left\rangle \right\langle e\right|\right)T^{1/2}_{n}=\left(1-\tau\right)\lambda_{n}\left|e\left\rangle \right\langle e\right|\oplus B_{0}.
\]
Thus the claim holds with $\lambda_{n+1}=\left(1-\tau\right)\lambda_{n}$.
By induction, $\lambda_{n}=\left(1-\tau\right)^{n}\lambda_{0}$, and
hence 
\[
T_{n}=\left(1-\tau\right)^{n}\lambda_{0}\left|e\left\rangle \right\langle e\right|\oplus B_{0}\to0\oplus B_{0}.
\]
Therefore $T_{\infty}=0\oplus B_{0}$. 
\end{proof}

\begin{rem}
Thus the weighted recursion already exhibits two opposite behaviors
on the stabilized support: complete collapse may occur, as in \prettyref{exa:2-4},
while nonzero transverse mass may persist unchanged when the transverse
subspace is reducing. The remaining difficulty lies in the  coupled
case, where $e^{\perp}$ does not reduce $T_{0}$.
\end{rem}

The block-decoupling identity (\prettyref{prop:3-1}) controls the
disappearance of the distinguished direction in the limit, but it
does not quantify how singularity develops. A complementary global
identity appears on the inverse side: once the active block is strictly
positive, the inverse dynamics becomes additive, with a rank-one increment
at each step. This yields a monotone blow-up mechanism for $\mathrm{tr}\left(T^{-1}_{n}\right)$
and a quantitative decay estimate for the smallest eigenvalue.
\begin{prop}
\label{prop:3-4} Consider \prettyref{eq:3-1}--\prettyref{eq:3-3}
with $T_{0}\in B\left(E\right)_{+}$ strictly positive. Then each
$T_{n}$ is strictly positive, and 
\[
T^{-1}_{n+1}=T^{-1}_{n}+\frac{\tau}{\rho}|T^{-1/2}_{n}e\left\rangle \right\langle T^{-1/2}_{n}e|.
\]
Set $\beta_{n}:=\left\langle e,T^{-1}_{n}e\right\rangle $ and $s_{n}:=\mathrm{tr}\left(T^{-1}_{n}\right)$.
Then 
\[
\beta_{n+1}-\beta_{n}=\frac{\tau}{\rho}|\langle e,T^{-1/2}_{n}e\rangle|^{2},
\]
and 
\[
s_{n+1}-s_{n}=\frac{\tau}{\rho}\beta_{n}.
\]

In particular, 
\[
\beta_{n}\ge\beta_{0}+n\frac{\tau}{\rho\left\Vert T_{0}\right\Vert }\qquad\left(n\ge0\right),
\]
and hence 
\[
s_{n}\ge s_{0}+\frac{\tau}{\rho}n\beta_{0}+\frac{\tau^{2}}{2\rho^{2}\left\Vert T_{0}\right\Vert }n\left(n-1\right)\qquad\left(n\ge0\right).
\]
Consequently, 
\[
\lambda_{\min}\left(T_{n}\right)\le\frac{\dim E}{s_{n}}\le\frac{C}{n^{2}}\qquad\left(n\ge1\right),
\]
for some constant $C>0$ depending only on $T_{0}$ and $\tau$. 
\end{prop}

\begin{proof}
Since 
\[
I_{E}-\tau\left|e\left\rangle \right\langle e\right|=\rho\left|e\left\rangle \right\langle e\right|+Q\ge\rho I_{E}>0,
\]
it follows inductively that every $T_{n}$ is strictly positive.

Now 
\[
\left(I_{E}-\tau\left|e\left\rangle \right\langle e\right|\right)^{-1}=Q+\rho^{-1}\left|e\left\rangle \right\langle e\right|=I_{E}+\frac{\tau}{\rho}\left|e\left\rangle \right\langle e\right|.
\]
Therefore 
\[
T^{-1}_{n+1}=T^{-1/2}_{n}(I_{E}-\tau\left|e\left\rangle \right\langle e\right|)^{-1}T^{-1/2}_{n}=T^{-1}_{n}+\frac{\tau}{\rho}|T^{-1/2}_{n}e\left\rangle \right\langle T^{-1/2}_{n}e|.
\]

Taking the matrix coefficient against $e$ gives 
\[
\beta_{n+1}-\beta_{n}=\frac{\tau}{\rho}|\langle e,T^{-1/2}_{n}e\rangle|^{2}.
\]
Taking traces gives 
\[
s_{n+1}-s_{n}=\frac{\tau}{\rho}\Vert T^{-1/2}_{n}e\Vert^{2}=\frac{\tau}{\rho}\beta_{n}.
\]

Since $0<T_{n+1}\le T_{n}\le T_{0}$, the function $t\mapsto t^{-1/2}$
being operator monotone decreasing on $\left(0,\infty\right)$ yields
\[
T^{-1/2}_{n}\ge T^{-1/2}_{0}\ge\left\Vert T_{0}\right\Vert ^{-1/2}I_{E}.
\]
Hence 
\[
\langle e,T^{-1/2}_{n}e\rangle\ge\left\Vert T_{0}\right\Vert ^{-1/2},
\]
and so 
\[
\beta_{n+1}-\beta_{n}\ge\frac{\tau}{\rho\left\Vert T_{0}\right\Vert }.
\]
Summing from $0$ to $n-1$ gives 
\[
\beta_{n}\ge\beta_{0}+n\frac{\tau}{\rho\left\Vert T_{0}\right\Vert }.
\]

Next, 
\[
s_{n}=s_{0}+\frac{\tau}{\rho}\sum^{n-1}_{k=0}\beta_{k}\ge s_{0}+\frac{\tau}{\rho}\sum^{n-1}_{k=0}\left(\beta_{0}+k\frac{\tau}{\rho\left\Vert T_{0}\right\Vert }\right),
\]
which gives 
\[
s_{n}\ge s_{0}+\frac{\tau}{\rho}n\beta_{0}+\frac{\tau^{2}}{2\rho^{2}\left\Vert T_{0}\right\Vert }n\left(n-1\right).
\]

Finally, 
\[
\mathrm{tr}\left(T^{-1}_{n}\right)\le\dim E\,\left\Vert T^{-1}_{n}\right\Vert =\frac{\dim E}{\lambda_{\min}\left(T_{n}\right)},
\]
hence 
\[
\lambda_{\min}\left(T_{n}\right)\le\frac{\dim E}{\mathrm{tr}\left(T^{-1}_{n}\right)}=\frac{\dim E}{s_{n}}.
\]
The quadratic lower bound for $s_{n}$ therefore implies 
\[
\lambda_{\min}\left(T_{n}\right)\le\frac{C}{n^{2}}
\]
for some $C>0$ depending only on $T_{0}$ and $\tau$. 
\end{proof}

\prettyref{prop:3-4} shows that the weighted recursion admits a dual
additive description on the inverse side. In particular, the active
block develops quantitative singularity: the trace of $T^{-1}_{n}$
grows at least quadratically, and therefore the smallest eigenvalue
decays at least on the order of $n^{-2}$. This complements the block-decoupling
from \prettyref{prop:3-1}, but it still does not determine the rank
of the limit.

In dimension two, however, the remaining geometry becomes rigid enough
to give a classification.

\section{Classification in Active Dimension 2}\label{sec:4}

When the active block has dimension $2$, the weighted recursion becomes
rigid enough to classify. In this case the orthogonal complement of
the distinguished direction is one-dimensional, so there is no room
for a  higher-rank stationary transverse component. As a result, the
asymptotic behavior is governed by a simple dichotomy: either the
transverse line is already reducing for the initial datum, in which
case it persists unchanged, or the active block is coupled, in which
case the iteration collapses completely.

The results of this section give a description of the limit in the
two-dimensional active case and, through the reduction established
earlier, yield the corresponding classification for the original finite-dimensional
recursion whenever the active block has dimension $2$. This is the
last dimension in which the limiting geometry is rigid; beyond it,
the family of stationary points supported on the transverse subspace
becomes  larger.
\begin{thm}[Complete classification in the weighted two-dimensional case]
\label{thm:4-1} Assume $\dim E=2$, let $T_{0}\in B\left(E\right)_{+}$
be strictly positive, and assume $u_{E}\ne0$. Write 
\[
e:=\frac{u_{E}}{\left\Vert u_{E}\right\Vert },\qquad\tau:=\left\Vert u_{E}\right\Vert ^{2}\in\left(0,1\right),\qquad\rho:=1-\tau\in\left(0,1\right),
\]
and let $f$ be a unit vector spanning $e^{\perp}$. With respect
to $E=\mathbb{C}e\oplus\mathbb{C}f$, write 
\[
T_{0}=\begin{pmatrix}a_{0} & b_{0}\\
\overline{b_{0}} & d_{0}
\end{pmatrix}.
\]
Then the following dichotomy holds. 
\begin{enumerate}
\item If $b_{0}=0$, equivalently if $e^{\perp}$ reduces $T_{0}$, then
$T_{\infty}=0\oplus d_{0}$. 
\item If $b_{0}\ne0$, then $T_{\infty}=0$. 
\end{enumerate}
\end{thm}

\begin{proof}
If $b_{0}=0$, then $e^{\perp}$ reduces $T_{0}$, so $T_{0}=a_{0}\left|e\left\rangle \right\langle e\right|\oplus d_{0}\left|f\left\rangle \right\langle f\right|$.
Since 
\[
I_{E}-\left|u_{E}\left\rangle \right\langle u_{E}\right|=I_{E}-\tau\left|e\left\rangle \right\langle e\right|=\rho\left|e\left\rangle \right\langle e\right|\oplus\left|f\left\rangle \right\langle f\right|,
\]
an induction gives 
\[
T_{n}=\rho^{n}a_{0}\left|e\left\rangle \right\langle e\right|\oplus d_{0}\left|f\left\rangle \right\langle f\right|.
\]
Hence 
\[
T_{\infty}=0\oplus d_{0}.
\]

Assume now that $b_{0}\ne0$. Replacing $f$ by a unimodular multiple,
we may assume that $b_{0}>0$. Then $T_{0}$ is real symmetric in
the basis $\left\{ e,f\right\} $, and since the recursion is defined
using the positive square root and the real diagonal matrix $\mathrm{diag}\left(\rho,1\right)$,
it follows inductively that every $T_{n}$ is real symmetric in this
basis. We claim that then $b_{n}>0$ for every $n\ge0$. Indeed, let
\[
T_{n}=\begin{pmatrix}a_{n} & b_{n}\\
b_{n} & d_{n}
\end{pmatrix},\qquad b_{n}>0.
\]
Since $T_{n}>0$, the square-root formula gives 
\[
T^{1/2}_{n}=\frac{T_{n}+\sqrt{\det T_{n}}\,I}{\sqrt{\mathrm{tr}\left(T_{n}\right)+2\sqrt{\det T_{n}}}}.
\]
Therefore 
\[
T_{n+1}=T^{1/2}_{n}\begin{pmatrix}\rho & 0\\
0 & 1
\end{pmatrix}T^{1/2}_{n},
\]
and a direct computation yields 
\[
b_{n+1}=b_{n}\frac{\rho a_{n}+\rho\sqrt{\det T_{n}}+d_{n}+\sqrt{\det T_{n}}}{\mathrm{tr}\left(T_{n}\right)+2\sqrt{\det T_{n}}}.
\]
The scalar factor is strictly positive, so $b_{n+1}>0$. By induction,
$b_{n}>0$, $n\ge0$. 

Now define 
\[
\xi_{n}:=\frac{b_{n}}{d_{n}},\qquad\zeta_{n}:=\frac{\sqrt{\det T_{n}}}{d_{n}}.
\]
Since $T_{n}>0$, one has $d_{n}>0$, $\xi_{n}>0$, and 
\[
a_{n}=d_{n}\left(\xi^{2}_{n}+\zeta^{2}_{n}\right).
\]
Using the square-root formula again, one obtains 
\[
d_{n+1}=d_{n}\frac{\rho\xi^{2}_{n}+\left(\zeta_{n}+1\right)^{2}}{\xi^{2}_{n}+\left(\zeta_{n}+1\right)^{2}},
\]
and 
\[
\xi_{n+1}=\xi_{n}\frac{\rho\xi^{2}_{n}+\rho\zeta^{2}_{n}+\rho\zeta_{n}+\zeta_{n}+1}{\rho\xi^{2}_{n}+\left(\zeta_{n}+1\right)^{2}}=\xi_{n}\left(1-\tau\frac{\zeta_{n}\left(\zeta_{n}+1\right)}{\rho\xi^{2}_{n}+\left(\zeta_{n}+1\right)^{2}}\right).
\]
In particular, 
\[
0<\xi_{n+1}\le\xi_{n}\qquad\left(n\ge0\right),
\]
so $\left(\xi_{n}\right)$ decreases to a limit $\xi_{\infty}\ge0$.

We now show that $d_{n}\to0$. Suppose, to the contrary, that $d_{\infty}:=\lim_{n\to\infty}d_{n}>0$.
By \prettyref{prop:2-3}, $\det T_{n}=\rho^{n}\det T_{0}$, hence
\[
\zeta_{n}=\frac{\sqrt{\det T_{n}}}{d_{n}}=\frac{\rho^{n/2}\sqrt{\det T_{0}}}{d_{n}}.
\]
Since $d_{n}\to d_{\infty}>0$, it follows that $\sum^{\infty}_{n=0}\zeta_{n}<\infty$.
Also, 
\[
\frac{\zeta_{n}\left(\zeta_{n}+1\right)}{\rho\xi^{2}_{n}+\left(\zeta_{n}+1\right)^{2}}\le\frac{\zeta_{n}\left(\zeta_{n}+1\right)}{\left(\zeta_{n}+1\right)^{2}}\le\zeta_{n},
\]
so $\xi_{n+1}\ge\xi_{n}\left(1-\tau\zeta_{n}\right)$. Since $\sum\zeta_{n}<\infty$,
there exists $N$ such that 
\[
0<1-\tau\zeta_{n}<1\qquad\left(n\ge N\right),
\]
and the infinite product $\prod^{\infty}_{n=N}\left(1-\tau\zeta_{n}\right)$
converges to a positive number. Therefore, for every $m>N$, 
\[
\xi_{m}\ge\xi_{N}\prod^{m-1}_{n=N}\left(1-\tau\zeta_{n}\right),
\]
and hence $\xi_{\infty}>0$. 

On the other hand, \prettyref{prop:3-1} gives $a_{n}\to0$. Since
$a_{n}=d_{n}\left(\xi^{2}_{n}+\zeta^{2}_{n}\right)$, and $d_{n}\to d_{\infty}>0$,
$\xi_{n}\to\xi_{\infty}>0$, this is impossible. Thus $d_{\infty}>0$
cannot occur, and therefore $d_{n}\to0$. 

Finally, \prettyref{prop:3-1} also gives $a_{n}\to0$, $b_{n}\to0$.
Together with $d_{n}\to0$, this yields $T_{n}\to0$. Hence $T_{\infty}=0$.
This proves (2). 
\end{proof}

\begin{cor}
\label{cor:4-2} Let $T_{0}\in B\left(E\right)_{+}$ be strictly positive.
Let $\mathcal{K}$ be the maximal $T_{0}$-reducing subspace contained
in $e^{\perp}$. Assume $\dim\mathcal{K}^{\perp}=2$. Then 
\[
T_{\infty}=T_{0}\restriction_{\mathcal{K}}\oplus0
\]
with respect to $E=\mathcal{K}\oplus\mathcal{K}^{\perp}$.
\end{cor}

\begin{proof}
Since $\mathcal{K}$ reduces $T_{0}$ and $\mathcal{K}\subseteq e^{\perp}$,
the same induction used earlier shows that $\mathcal{K}$ reduces
every $T_{n}$, and 
\[
T_{n}\restriction_{\mathcal{K}}=T_{0}\restriction_{\mathcal{K}}\qquad\left(n\ge0\right).
\]
Hence, with respect to $E=\mathcal{K}\oplus\mathcal{K}^{\perp}$,
one has 
\[
T_{n}=T_{0}\restriction_{\mathcal{K}}\oplus S_{n},
\]
where $\left(S_{n}\right)$ is a weighted recursion on the two-dimensional
space $\mathcal{K}^{\perp}$: 
\[
S_{n+1}=S^{1/2}_{n}\left(I_{\mathcal{K}^{\perp}}-\tau\left|e\left\rangle \right\langle e\right|\right)S^{1/2}_{n}.
\]
Here $e\in\mathcal{K}^{\perp}$, because $\mathcal{K}\subseteq e^{\perp}$.

Now $\dim\mathcal{K}^{\perp}=2$, and $\mathcal{K}^{\perp}\cap e^{\perp}$
is one-dimensional. If this line reduced $S_{0}=T_{0}\restriction_{\mathcal{K}^{\perp}}$,
then its direct sum with $\mathcal{K}$ would be a larger $T_{0}$-reducing
subspace contained in $e^{\perp}$, contradicting the maximality of
$\mathcal{K}$. Therefore the transverse line in $\mathcal{K}^{\perp}$
does not reduce $S_{0}$.

By \prettyref{thm:4-1}, it follows that $S_{\infty}=0$. Thus $T_{\infty}=T_{0}\restriction_{\mathcal{K}}\oplus0$. 
\end{proof}

The two-dimensional classification shows that a  coupled active block
collapses completely when its dimension is two. Beyond dimension two,
however, nonzero limits may survive. The next proposition identifies
the stationary points of the weighted recursion and isolates the precise
obstruction to complete collapse in higher dimension.
\begin{prop}
\label{prop:4-3} Let 
\[
\Phi\left(S\right):=S^{1/2}\left(I_{E}-\tau\left|e\left\rangle \right\langle e\right|\right)S^{1/2},\qquad S\in B\left(E\right)_{+},
\]
where $0<\tau<1$ and $\left\Vert e\right\Vert =1$. For $S\in B\left(E\right)_{+}$,
the following are equivalent: 
\begin{enumerate}
\item $\Phi\left(S\right)=S$. 
\item $S^{1/2}e=0$. 
\item $Se=0$. 
\item $\mathrm{ran}\left(S\right)\subseteq e^{\perp}$. 
\end{enumerate}
In particular, every positive operator supported on $e^{\perp}$ is
a stationary point of the weighted recursion.

Consequently, the limit $T_{\infty}$ of the weighted recursion is
always a stationary point, and therefore 
\[
T_{\infty}e=0,\qquad\mathrm{ran}\left(T_{\infty}\right)\subseteq e^{\perp}.
\]
\end{prop}

\begin{proof}
For $S\in B\left(E\right)_{+}$, one has 
\[
\Phi\left(S\right)=S^{1/2}\left(I_{E}-\tau\left|e\left\rangle \right\langle e\right|\right)S^{1/2}=S-\tau\left|S^{1/2}e\left\rangle \right\langle S^{1/2}e\right|.
\]
Hence 
\[
\Phi\left(S\right)=S\iff|S^{1/2}e\left\rangle \right\langle S^{1/2}e|=0\iff S^{1/2}e=0.
\]
Thus (1) and (2) are equivalent.

Since $S\ge0$, one has 
\[
\ker(S^{1/2})=\ker\left(S\right),
\]
so (2) and (3) are equivalent.

Next, 
\[
Se=0\iff\left\langle Sx,e\right\rangle =0\quad\text{for all }x\in E\iff\left\langle x,Se\right\rangle =0\quad\text{for all }x\in E,
\]
and because $S=S^{*}$, this is equivalent to 
\[
\left\langle Sx,e\right\rangle =0\quad\text{for all }x\in E,
\]
that is, 
\[
\mathrm{ran}\left(S\right)\subseteq e^{\perp}.
\]
So (3) and (4) are equivalent.

This proves the equivalence of (1)-(4). The statement about stationary
points supported on $e^{\perp}$ is exactly (4)$\Rightarrow$(1).

Finally, since $T_{n+1}=\Phi\left(T_{n}\right)$ and $T_{n}\to T_{\infty}$
in norm, continuity of the map $S\mapsto S^{1/2}$ on $B\left(E\right)_{+}$
shows that $\Phi\left(T_{\infty}\right)=T_{\infty}$. Hence $T_{\infty}$
is a stationary point, and the last assertion follows from the equivalence
already proved. 
\end{proof}

\begin{rem}
\prettyref{prop:4-3} shows that the higher-dimensional problem is
not merely a larger version of the two-dimensional one. In dimension
two, the space $e^{\perp}$ is one-dimensional, so the cone of stationary
points supported on $e^{\perp}$ is itself one-dimensional. This rigidity
is  what makes complete classification possible.

When $\dim E\ge3$, the cone of stationary points supported on $e^{\perp}$
has  higher-dimensional geometry. Thus nonzero limits are no longer
excluded by fixed-point considerations alone, and the problem becomes
one of determining which transverse stationary operator is selected
by the nonlinear iteration. 

\bibliographystyle{plain}
\bibliography{ref}

\begin{thebibliography}{10}

\bibitem{MR242573}
W.~N. Anderson, Jr. and R.~J. Duffin.
\newblock Series and parallel addition of matrices.
\newblock {\em J. Math. Anal. Appl.}, 26:576--594, 1969.

\bibitem{MR356949}
W.~N. Anderson, Jr. and G.~E. Trapp.
\newblock Shorted operators. {II}.
\newblock {\em SIAM J. Appl. Math.}, 28:60--71, 1975.

\bibitem{MR287970}
William~N. Anderson, Jr.
\newblock Shorted operators.
\newblock {\em SIAM J. Appl. Math.}, 20:520--525, 1971.

\bibitem{MR2234254}
Jorge Antezana, Gustavo Corach, and Demetrio Stojanoff.
\newblock Spectral shorted operators.
\newblock {\em Integral Equations Operator Theory}, 55(2):169--188, 2006.

\bibitem{MR2284176}
Rajendra Bhatia.
\newblock {\em Positive Definite Matrices}.
\newblock Princeton Series in Applied Mathematics. Princeton University Press,
  Princeton, NJ, 2007.

\bibitem{MR2573240}
Brian Blackwood, S.~K. Jain, K.~M. Prasad, and Ashish~K. Srivastava.
\newblock Shorted operators relative to a partial order in a regular ring.
\newblock {\em Comm. Algebra}, 37(11):4141--4152, 2009.

\bibitem{MR2580440}
A.~B\"{o}ttcher and I.~M. Spitkovsky.
\newblock A gentle guide to the basics of two projections theory.
\newblock {\em Linear Algebra Appl.}, 432(6):1412--1459, 2010.

\bibitem{MR938493}
C.~A. Butler and T.~D. Morley.
\newblock A note on the shorted operator.
\newblock {\em SIAM J. Matrix Anal. Appl.}, 9(2):147--155, 1988.

\bibitem{MR852902}
Hans Goller.
\newblock Shorted operators and rank decomposition matrices.
\newblock {\em Linear Algebra Appl.}, 81:207--236, 1986.

\bibitem{MR24574}
M.~Krein.
\newblock The theory of self-adjoint extensions of semi-bounded {H}ermitian
  transformations and its applications. {I}.
\newblock {\em Rec. Math. [Mat. Sbornik] N.S.}, 20(62):431--495, 1947.

\bibitem{MR563399}
Fumio Kubo and Tsuyoshi Ando.
\newblock Means of positive linear operators.
\newblock {\em Math. Ann.}, 246(3):205--224, 1979/80.

\bibitem{MR1465881}
Volker Metz.
\newblock Shorted operators: an application in potential theory.
\newblock {\em Linear Algebra Appl.}, 264:439--455, 1997.

\bibitem{MR401352}
Katsuyoshi Nishio and Tsuyoshi Ando.
\newblock Characterizations of operations derived from network connections.
\newblock {\em J. Math. Anal. Appl.}, 53(3):539--549, 1976.

\bibitem{MR3145756}
Evgeniy Pustylnik and Simeon Reich.
\newblock Infinite products of arbitrary operators and intersections of
  subspaces in {H}ilbert space.
\newblock {\em J. Approx. Theory}, 178:91--102, 2014.

\bibitem{MR2903120}
Evgeniy Pustylnik, Simeon Reich, and Alexander~J. Zaslavski.
\newblock Convergence of non-periodic infinite products of orthogonal
  projections and nonexpansive operators in {H}ilbert space.
\newblock {\em J. Approx. Theory}, 164(5):611--624, 2012.

\bibitem{MR493419}
Michael Reed and Barry Simon.
\newblock {\em Methods of modern mathematical physics. {I}. {F}unctional
  analysis}.
\newblock Academic Press, New York-London, 1972.

\bibitem{MR4310540}
Simeon Reich and Rafa\l Zalas.
\newblock Error bounds for the method of simultaneous projections with
  infinitely many subspaces.
\newblock {\em J. Approx. Theory}, 272:Paper No. 105648, 24, 2021.

\bibitem{tian2025alt}
James Tian.
\newblock Alternating weighted residual flows and the non-commutative gap.
\newblock {\em arXiv.2512.06968}, 2025.

\bibitem{tian2025wr}
James Tian.
\newblock Residual-weighted decomposition of positive operators.
\newblock {\em arXiv.2512.00167}, 2025.

\end{thebibliography}
\end{rem}

\end{document}